\documentclass[lineno]{biometrika}

\usepackage{amsfonts,amsmath}

\usepackage{times}
\usepackage{bm}
\usepackage{natbib}

\usepackage[plain,noend]{algorithm2e}

\makeatletter
\renewcommand{\algocf@captiontext}[2]{#1\algocf@typo. \AlCapFnt{}#2} 
\def\@algocf@capt@plain{top}
\renewcommand{\algocf@makecaption}[2]{%
  \addtolength{\hsize}{\algomargin}%
  \sbox\@tempboxa{\algocf@captiontext{#1}{#2}}%
  \ifdim\wd\@tempboxa >\hsize
    \hskip .5\algomargin%
    \parbox[t]{\hsize}{\algocf@captiontext{#1}{#2}}
  \else%
    \global\@minipagefalse%
    \hbox to\hsize{\box\@tempboxa}
  \fi%
  \addtolength{\hsize}{-\algomargin}%
}
\makeatother

\begin{document}

\jname{Biometrika}
\jyear{2016}
\jvol{103}
\jnum{1}
\accessdate{Advance Access publication on 31 July 2016}
\copyrightinfo{\Copyright\ 2016 Biometrika Trust\goodbreak {\em Printed in Great Britain}}

\received{January 2016}
\revised{March 2016}

\markboth{G. V. Chavez \and R. Kleeman}{Near-Gaussian densities and entropy}

\title{Near-Gaussian entropic functional calculation and density estimation using an asymptotic series}

\author{G. V. CHAVEZ}
\affil{B.A. candidate (Mathematics) at NYU College of Arts and Sciences, Courant Institute of Mathematical Sciences New York 10012, USA \email{gvc214@nyu.edu}} 

\author{\and R. KLEEMAN}
\affil{Department of Mathematics, Courant Institute of Mathematical Sciences New York 10012, USA  \email{kleeeman@cims.nyu.edu}}

\maketitle

\begin{abstract}
Near-Gaussian probability densities are common in many important physical applications. Here we develop an asymptotic expansion methodology for computing entropic functionals for such densities. The expansion proposed is a close relative of standard perturbation expansions in quantum field theory. We give novel results on the low-order effects of non-Gaussian even moments and asymmetry (e.g. skewness) on the entropy. The asymptotic expansion is also used to define a best fit maximum entropy density given a set of observed low order moments. The maximum entropy density estimation technique consists simply of the solution of a small set of algebraic equations and is therefore more straightforward numerically than classical maximum-entropy methods which rely on sophisticated convex optimization techniques.
\end{abstract}


\section{Introduction}

The Gaussian assumption for probability densities is very widespread in many applications of mathematical statistics. This simply reflects the reality that many useful densities are often close to Gaussian and the assumption made facilitates practical calculation. Some important examples of this situation include the Kalman filter in data science and the densities often seen in turbulent fluid systems (see as representative examples the following work by the second author and co-workers, \cite{klmati02},
\cite{TCT15}, \cite{klee04d} and \cite{kleeman05a}).
 
Given this situation, it is natural to consider corrections to the standard Gaussian results using the framework of perturbation theory. This approach is actually the standard one followed in interacting quantum field theory (see e.g. \cite{schulman2012techniques}) and results in perturbative correction terms represented by Feynman diagrams. The method used is mathematically an asymptotic expansion using small parameters which control deviations from Gaussianity. 

In this contribution we explain the asymptotic expansion utilized and then use it to calculate two things which often present practical difficulties when general densities are considered: Entropic functionals and maximum entropy density estimates. In the first case, direct calculation methods require some kind of often arbitrary coarse graining of the random variable space (see e.g. \cite{kle05}) while in the second case intricate methods from convex optimization are often required (see e.g. \cite{Abr10}). Of course the present methodology, unlike earlier techniques, is restricted to near-Gaussian densities but given their ubiquity it is of clear utility.

As one might imagine there are many possible practical applications for the methodology to be presented here. In this paper we consider simple perturbatively non-Gaussian cases that give novel results on the impact of higher-order, non-Gaussian moments on the entropy. The application to more complex cases should however be conceptually clear and will be dealt with specifically in future work.  

In the next section we give a general outline of the proposed methodology. In succeeding sections we consider specific examples and clearly illustrate the method. In section 3, we demonstrate our method in the evaluation of a first order approximation of the moments and entropy of two near-Gaussian probability densities and show that the effects of asymmetry on the probability density cannot be seen at first order in the entropy. In section 4, we proceed to a second order approximation of the moments and entropy which reveals the low order effects of asymmetry which lead to an increase in the entropy. Section 5 illustrates the maximum entropy method while Section 6 contains a discussion and possible future work. 

\section{Method outline} 
Suppose we have a near-Gaussian probability density defined on $\mathbb{R}^{n}$
of the form
$$
p(\vec{x})=\frac{1}{Z}e^{ -q(\vec{x})-\varepsilon h(\vec{x})}
$$
where $\vec{x}\in \mathbb{R}^{n}$, $q(\vec{x})$ is a quadratic function, and $\varepsilon \ll 1$
is a perturbation parameter. For simplicity we consider for the present
only one such parameter but the generalization to many is clear.

If we expand $\exp(-\varepsilon h(\vec{x}))$ as a McClaurin series in $\varepsilon$
then we can write $Z$ as the following expectation
$$
Z=Z_{G} \int ... \int p_{G}(\vec{x})\left( 1-\varepsilon h(\vec{x})+\frac{\varepsilon^{2}}{2}h^{2}(\vec{x})-\ldots\right)dx_{1},...,dx_{n}$$
$$p_{G}(\vec{x}) = \frac{e^{-q(\vec{x})}}{Z_{G}}
$$
where $p_{G}(\vec{x})$ is the associated Gaussian density. Now if we further
assume that the function $h$ is a polynomial of some kind then we
can clearly write $Z$ as a power series in $\varepsilon$ with coefficients
which are Gaussian moments and thus known by Isserlis' theorem. Moments
with respect to $p$ may now be evaluated in a similar fashion but
using the already evaluated $Z$. Thus, for example, the first moment
is
$$
\mu_{(1)}=\frac{Z_{g}}{Z}\int ... \int  p_{G}(\vec{x})\left( 1-\varepsilon \vec{x}h(\vec{x})+\frac{\varepsilon^{2}\vec{x}h^{2}(\vec{x})}{2}-...\right)dx_{1}...dx_{n}$$
which can be evaluated again using Isserlis' theorem. Once the moments
are determined it is a straightforward matter to then evaluate entropic
functionals. For example, the Shannon entropy for the near-Gaussian density from above is 
$$
H=-\int ... \int p(\vec{x})\log p(\vec{x})dx_{1}...dx_{n}=\left\langle q\right\rangle _{p}+\varepsilon\left\langle h\right\rangle _{p}+\log Z
$$
which requires just two polynomial moments and $Z$.

Naturally for practical applications the various perturbation series
are truncated at some order and approximate quantities result. As
we shall see these series can be regarded as asymptotic expansions
which means that the errors of the approximations made can be bounded.
More detail on this may be found below. 

A conceptually similar methodology to the above is used widely in
quantum field theory although there one deals with the more technically
difficult situation of path integrals and $\mathbb{R}^{n}$ is extended to
an infinite dimensional vector space. Due to this infinite extension,
issues of divergence are common and require renormalization theory.
The various terms in the perturbation series there are commonly associated
with various Feynman diagrams.

Given that we have algebraic expressions involving the perturbation
expansion parameters for various moments of interest, one can reverse
the logic of the above derivation as follows: Prescribe a certain
set of $n$ moments which might arise from an experimental sample.
The associated maximum entropy density with such moments is given
by
\[
p(x)=\frac{1}{Z(\vec{\theta})}\exp\left(-\sum_{i=1}^{n}\theta_{i}r_{i}(x)\right)
\]
 where the $r_{i}$ are the polynomials associated with each moment.
The problem here then, as shown by \cite{mead84}, is to determine the coefficients $\theta_{i}$
through a Legendre transform of the moments . In the case that
these are small for $i>2$ we can use the above perturbation series
to derive algebraic equations for the coefficients in terms of the
prescribed moments. The solution gives then an approximation to the
maximum entropy density applicable to the prescribed moments. Of course
one then needs to check that this solution is self consistent with
the near-Gaussian assumption. 

\section{First-Order Approximations}

In this section we will demonstrate an approximation of the entropy of two, perturbatively non-Gaussian, probability density functions, using a first order approximation of the non-Gaussian moments. The first example will be a maximum-entropy density with an even, non-Gaussian, $p$th order moment. The second example have asymmetry about 0, in the form of a nonzero, odd, $q$th moment, where $q \geq 3$, as well as a non-Gaussian $p$th moment. We will show, using these two cases, that to study the contributions to the entropy from the asymmetry, a second order approximation is required.

\subsection{$p$th-Moment Perturbed Gaussian}
Consider a non-Gaussian probability density function of the form 
\begin{equation}
p(x)=\frac{C(\varepsilon)}{\sqrt{2 \pi}\sigma}e^{-\frac{x^{2}}{2 \sigma^{2}}}e^{-\varepsilon x^{p}}\end{equation}
where $p$ is an even positive integer, $C(\varepsilon)$ is a normalization function, and $\varepsilon$ is our non-Gaussianity parameter. This probability density is in the exponential family and hence is an entropy maximizing density, given a specified 1st (here zero), 2nd, and $p$th order non-Gaussian moment. 
\begin{theorem}
The moments of (1) may be written as the infinite series 
\begin{equation}
\mu_{(k)}=C(\varepsilon)\sum_{n=0}^{\infty} \frac{(-\varepsilon)^{n}}{n!} (np+k-1)!! \sigma^{np+k},\end{equation}
where 
\begin{equation}C(\varepsilon)=\left(\sum_{n=0}^{\infty} \frac{(-\varepsilon)^{n}}{n!} (np-1)!! \sigma^{np}\right)^{-1}, \end{equation}
and $!!$ indicates the double factorial, defined as $k!!=k(k-2)(k-4)(k-6)...1$.
\end{theorem}
\begin{proof}
We use the Taylor expansion of the exponential function to write (1) as an infinite sum,
\begin{equation}
p(x)=\frac{C(\varepsilon)}{\sqrt{2\pi} \sigma} e^{-\frac{x^{2}}{2 \sigma^{2}}} \sum_{n=0}^{\infty} \frac{(-\varepsilon x^{p})^{n}}{n!}. \end{equation}
We will use this formula to compute the raw moments $\mu_{(k)}$ of (1). We write using (4)
$$\mu_{(k)}=\int_{-\infty}^{\infty} \frac{C(\varepsilon)}{\sqrt{2\pi} \sigma} e^{-\frac{x^{2}}{2 \sigma^{2}}}\sum_{n=0}^{\infty} \frac{(-\varepsilon)^{n}}{n!} x^{np+k}dx.$$
Hence rearranging and taking the integral, we see that 
\begin{equation}
\mu_{(k)}=C(\varepsilon)\sum_{n=0}^{\infty} \frac{(-\varepsilon)^{n}}{n!} \mu_{(np+k)}^{G},\end{equation}
where 
$$\mu_{(k)}^{G}=\int_{-\infty}^{\infty} \frac{1}{\sqrt{2 \pi} \sigma} e^{-\frac{x^{2}}{2 \sigma^{2}}}x^{k} dx$$
are the Gaussian central moments. 

Note that the raw moments in (5) also are dependent on the normalization function $C(\varepsilon)$. 
To compute the function $C(\varepsilon)$, we write 
\begin{equation}\int_{-\infty}^{\infty} p(x)dx=C(\varepsilon) \int_{\infty}^{\infty} \frac{1}{\sqrt{2\pi} \sigma} e^{-\frac{x^{2}}{2 \sigma^{2}}} e^{-\varepsilon x^{p}} dx=1.\end{equation}
We can then apply equation (4), and integrate to write 
\begin{equation}
1=C(\varepsilon) \left(1-\varepsilon \mu_{(p)}^{G}+\frac{\varepsilon^{2}}{2} \mu_{(2p)}^{G} -...\right)=C(\varepsilon)  \sum_{n=0}^{\infty} \frac{(-\varepsilon)^{n}}{n!} \mu_{(np)}^{G}. \end{equation}
We may write the normalization function $C(\varepsilon)$ in terms of the non-Gaussianity parameter $\varepsilon$ and the Gaussian central moments by rearranging (7). 

To write the series in (5) and (7) more explicitly, we note the expression for the even central moments of a Gaussian probability density function,
\begin{equation}
\mu_{(k)}^{G}=\sigma^{k}(k-1)!! \end{equation}
Substituting (8) into (5) gives the result in (2), and substituting (8) into (7) gives (3).
\end{proof}
The infinite series in (2) and (3) are proven to be asymptotic series in the supplementary matrials. We now have the necessary expressions for computation of the entropic functional. Now we consider the case of interest, the \textit{near-Gaussian} or \textit{perturbatively non-Gaussian} case, which is when the non-Gaussianity parameter is small relative to the Gaussian variance. That is, the case when 
\begin{equation} \varepsilon \sigma^{p} \ll 1. \end{equation}
In this case, we may truncate our expansions for $\mu_{(k)}$ and $C(\varepsilon)$ at some sufficiently low $n$. This gives a low-order approximation of the moments and normalization function. The supplementary materials prove the following theorem, which allows us to bound the error in our moment approximations.

\begin{theorem}
The series in (2) and (3) satisfy the following inequality.
\begin{eqnarray}
\left|\Delta_{k}(N)\right| & = & \left|\frac{\mu_{(k)}}{C(\varepsilon)}-\sum_{n=0}^{N} \frac{(-\varepsilon)^{n}}{n!}(np+k-1)!!\sigma^{np+k} \right|\nonumber\\
 & \leq & \frac{\varepsilon^{N+1}}{(N+1)!}(Np+k-1+p)!!\sigma^{(N+1)p+k}\label{eq:10}
\end{eqnarray}
where $\Delta_{k} (N)$ indicates the error in the $k$th asymptotic series when truncated at order $N$.
\end{theorem}

We can multiply both sides of (10) by $C(\varepsilon)$ and use the above inequality to get the following upper bound on our moment approximations' error.
\begin{eqnarray}
\left|\Delta \mu_{(k)}(N)\right| & = & \left|\mu_{(k)}-C(\varepsilon)\sum_{n=0}^{N} \frac{(-\varepsilon)^{n}}{n!}(np+k-1)!!\sigma^{np+k}\right|\nonumber\\
& \leq & C(\varepsilon)\frac{\varepsilon^{N+1}}{(N+1)!}(Np+k-1+p)!!\sigma^{(N+1)p+k}\label{eq:11}
\end{eqnarray}
where $\Delta \mu{(k)}(N)$ is the error in the asymptotic series' moment approximation when truncated at order $N$. We will now use Theorems 1 and 2, as well as equation (11) to prove a new result about the entropy.

\begin{theorem}
A 1st order approximation of (1)'s moments and normalization function using (2) and (3) gives the following 1st order result for (1)'s entropy.
\begin{equation}
\hspace{-.6cm}H_{p} \approx \log \sqrt{2\pi} \sigma \left(1-\varepsilon \sigma^{p}(p-1)!! \right) +\frac{1}{1-\varepsilon \sigma^{p} (p-1)!!} \left[\frac{1}{2}+\varepsilon \sigma^{p}\left((p-1)!!-\frac{1}{2} (p+1)!!\right) \right]
\end{equation}
and this approximation has the following upper bound on its error.
\begin{equation}
\left|\Delta H_{p}(1) \right| \leq \log\left(1+\beta_{p}\frac{\varepsilon^{2}}{2}(2p-1)!!\sigma^{2p} \right)+\beta_{p}\frac{\varepsilon^{2}}{4}(2p+1)!!\sigma^{2p}+\beta_{p} \frac{\varepsilon^{3}}{2}(3p-1)!!\sigma^{3p},
\end{equation}
where 
\begin{equation}
\beta_{p}=\left(1-\varepsilon (p-1)!!\sigma^{p}-\frac{\varepsilon^{2}}{2}(2p-1)!!\sigma^{2p}\right)^{-1}.
\end{equation}

\end{theorem}
\begin{proof}
We first prove the result in (12). A first order approximation for the function $C(\varepsilon)$ using (7) gives
\begin{equation}
C(\varepsilon) \approx \frac{1}{1- \varepsilon \mu_{(p)}^{G}}, \end{equation}
and hence a first order approximation for the raw moments using (5) and (15) gives,
\begin{equation}
\mu_{(k)} \approx  \frac{\mu_{(k)}^{G}-\varepsilon\mu_{(p+k)}^{G}}{1-\varepsilon \mu_{(p)}^{G}}.\end{equation}
We will use these first order approximations below for for the entropy. 

Recall that for a continuous probability density function $p(x)$, the differential entropy is defined
\begin{equation}
H=-\int_{-\infty}^{\infty} p(x) \log p(x) dx. \end{equation}
We substitute (1) into (17) to get (1)'s entropy. To begin, we calculate the negative logarithm of (1), which we write as
\begin{equation}
-\log p(x)=\frac{1}{2} \log \frac{2 \pi \sigma^{2}}{C^{2}(\varepsilon)}+\frac{x^{2}}{2 \sigma^{2}} +\varepsilon x^{p}. \end{equation}
Multiplying (18) by $p(x)$ and integrating gives us the expression for the entropy, 
\begin{equation}
H=\frac{1}{2} \log \frac{2 \pi \sigma^{2}}{C^{2}(\varepsilon)}+\frac{\mu_{(2)}}{2 \sigma^{2}} + \varepsilon \mu_{(p)}. \end{equation}
Applying (15) and (16) for $k=2$ and $k=p$, and substituting into (19) gives 
$$H \approx \frac{1}{2} \log 2 \pi \left[\sigma \left(1- \varepsilon \mu_{(p)}^{G}\right) \right]^{2} + \frac{1}{2 \sigma^{2}} \frac{\sigma^{2}-\varepsilon \mu_{(p+2)}^{G}}{1-\varepsilon \mu_{(p)}^{G}}+\varepsilon \frac{\mu_{(p)}^{G}-\varepsilon \mu_{(2p)}^{G}}{1-\varepsilon \mu_{(p)}^{G}},$$
and we can write this expression in the form 
\begin{equation}
H \approx \log \sqrt{2 \pi} \sigma \left(1- \varepsilon \mu_{(p)}^{G}\right)+ \frac{1}{1-\varepsilon \mu_{(p)}^{G}}\left(\frac{1}{2}-\varepsilon \frac{\mu_{(p+2)}^{G}}{2 \sigma^{2}}+\varepsilon \mu_{(p)}^{G}-\varepsilon^{2} \mu_{(2p)}^{G} \right).\end{equation}
We are making a first order approximation. So we neglect (20)'s terms that are quadratic in $\varepsilon$. This gives us
\begin{equation}
H \approx \log \sqrt{2 \pi} \sigma \left(1- \varepsilon \mu_{(p)}^{G}\right)+\frac{1}{1-\varepsilon \mu_{(p)}^{G}} \left(\frac{1}{2}+\varepsilon \mu_{(p)}^{G}-\frac{\varepsilon}{2 \sigma^{2}} \mu_{(p+2)}^{G} \right). \end{equation}
Substituting (8) for the Gaussian moments, simplifying, and grouping terms gives the result in (12).

To prove (13) and (14), we first recall the general expression for (1)'s entropy in (19), and subtract this from an entropy expression with the approximated moments. This gives, after some simplification,
\begin{equation}
\Delta H_{p} (1)=\log \left(1 \pm C(\varepsilon)\left|\Delta_{0}(1)\right| \right) \pm \frac{1}{2\sigma^{2}}\left|\Delta \mu_{(2)}(1)\right| \pm \varepsilon \left|\Delta \mu_{(p)}(1) \right|.
\end{equation}
The maximum value of (22) must have the form
\begin{equation}\log\left(1 + C(\varepsilon)\left|\Delta_{0}(1)\right| \right) + \frac{1}{2\sigma^{2}}\left|\Delta \mu_{(2)}(1)\right| + \varepsilon \left|\Delta \mu_{(p)}(1) \right|.\end{equation}
Next note that from (10) and (11), we may substite in for the errors in (23), and simplify to write the upper bound
\begin{equation}
\left|\Delta H_{p} (1) \right| \leq \log\left(1 + C(\varepsilon)\frac{\varepsilon^{2}}{2}(2p-1)!!\sigma^{2p} \right) + C(\varepsilon)\frac{\varepsilon^{2}}{4}(2p+1)!!\sigma^{2p} +C(\varepsilon)\frac{\varepsilon^{3}}{2}(3p-1)!!\sigma^{3p}.\end{equation}

Next, we show that (14)'s $\beta_{p}$ is an upper bound for $C(\varepsilon)$. To do this, we note from (10) that 
$$
\frac{1}{C(\varepsilon)}=1-\varepsilon (p-1)!! \sigma^{p} \pm \left|\Delta_{0}(1)\right|,$$
 and hence from the upper bound on $\left|\Delta_{0}(1)\right|$ in (10), the minimum value of $1/C(\varepsilon)$ is given by
\begin{equation}\frac{1}{C(\varepsilon)} \geq 1-\varepsilon (p-1)!! \sigma^{p}-\frac{\varepsilon^{2}}{2}(2p-1)!!\sigma^{2p}.\end{equation}
Rearranging (25)  gives $\beta_{p}$ as an upper bound for $C(\varepsilon)$. Substituting this into (24) completes the proof.
\end{proof}

Now let us consider (12) , (13), and (14) with $p=4$. This gives
\begin{equation}
H_{4} \approx \log \sqrt{2\pi} \sigma \left[1-3\varepsilon \sigma^{4} \right]+\frac{1}{2} \frac{1-9 \varepsilon \sigma^{4}}{1-3\varepsilon \sigma^{4}}, \end{equation}
with the error bound
\begin{equation}
\left|\Delta H_{4}(1) \right| \leq \log\left(1+\beta_{4} \frac{\varepsilon^{2}7!!}{2}\sigma^{8} \right)+\beta_{4} \frac{\varepsilon^{2}9!!}{4}\sigma^{8}+\beta_{4}\frac{\varepsilon^{3}11!!}{2}\sigma^{12},\end{equation}
and
\begin{equation}
\beta_{4}=\left(1-3\varepsilon \sigma^{p}-\frac{\varepsilon^{2}7!!}{2} \sigma^{8} \right)^{-1}. 
\end{equation}
These results are graphed in Fig. 1 for $\sigma=1$. Note that the entropy for the Gaussian component in (1) is easily shown to be
\begin{equation}
H^{G}=\log \sqrt{2\pi} \sigma+\frac{1}{2}. \end{equation}
By comparing (26) to (29), and noting (9), we can see that clearly the result in (26) is less than the entropy for the Gaussian component. Hence to first order, a 4th moment decreases the entropy result. Examination of (12) and comparison to (29) shows that this is true for general $p$. Hence the specification of a higher-order, even, non-Gaussian moment decreases the entropy from the corresponding Gaussian case at first order.

\subsection{Odd $q$th and Even $p$th-Moment Perturbed Gaussian}
Consider an asymmetric, non-Gaussian probability density function with the following form,
\begin{equation}
p(x)=\frac{C(\vec{\varepsilon})}{\sqrt{2 \pi} \sigma} e^{-\frac{x^{2}}{2 \sigma^{2}}} e^{\varepsilon_{q} x^{q}-\varepsilon_{p} x^{p}}, \end{equation}
where $p$ is even, $3 \leq q <p$ is odd, $\vec{\varepsilon}=(\varepsilon_{q},\varepsilon_{p})$, and $\varepsilon_{p}>0$, but $\varepsilon_{q}$ can be positive or negative. An $\varepsilon_{1}$ asymmetry parameter may always be removed with an appropriate translation (entropy is independent of the origin), so we neglect it here. So this is a non-Gaussian probability density function with asymmetry about 0 corresponding to the specification of a non-zero, odd $q$th moment as well as an even $p$th moment. 

One can use methods similar to those in Theorem 1 and the supplementary materials to show that the corresponding asymptotic series for the moments and normalization function have the form 
\begin{equation}
\mu_{(k)}=C(\vec{\varepsilon})\int_{-\infty}^{\infty}\frac{e^{-\frac{x^{2}}{2\sigma^{2}}}}{\sqrt{2 \pi}\sigma} x^{k}\sum_{n=0}^{\infty} \frac{1}{n!} \left(\varepsilon_{q}x^{q}-\varepsilon_{p}x^{p} \right)^{n}dx.
\end{equation}

We maintain a perturbative assumption, and for a first order approximation we neglect all terms with factors of the form $\varepsilon_{k}\varepsilon_{j}\sigma^{k+j}$.

\begin{theorem}
A first order approximation of (30)'s moments and normalization function using (31) gives an equivalent entropy result to (12).
\end{theorem}
\begin{proof}
A first order approximation of $C(\vec{\varepsilon})$ using (31) with $k=0$ gives
\begin{equation}
C(\vec{\varepsilon}) \approx \frac{1}{1-\varepsilon_{p}\mu_{(p)}^{G}}. \end{equation}
Note that this is equivalent to (15). The first order approximation for the raw moments using (31) gives
\begin{equation}\mu_{(k)} \approx C(\vec{\varepsilon})  \left( \mu_{(k)}^{G}+ \varepsilon_{q} \mu_{(k+q)}^{G}-\varepsilon_{p}\mu_{(k+p)}^{G}\right).\end{equation}
We next substitute (30) into (17), which gives (30)'s entropy,
\begin{equation}
H=\frac{1}{2} \log \frac{2 \pi \sigma^{2}}{C^{2}(\vec{\varepsilon})}+\frac{ \mu_{(2)}}{2\sigma^{2}} - \varepsilon_{q} \mu_{(q)}+\varepsilon_{p} \mu_{(p)}. \end{equation}
Applying (33) for the moments in (34), and grouping terms gives
\begin{equation}
H \approx \log \frac{\sqrt{2 \pi} \sigma}{C(\vec{\varepsilon})}+C(\vec{\varepsilon}) \left(\frac{1}{2}+\varepsilon_{p}\mu_{(p)}^{G}-\frac{\varepsilon_{p}}{2\sigma^{2}}\mu_{(p+2)}^{G}-\varepsilon_{q}^{2}\mu_{(2q)}^{G}-\varepsilon_{p}^{2}\mu_{(2p)}^{G} \right).
\end{equation}
We drop all terms of quadratic order in $\vec{\varepsilon}$ from (35), and hence our first order entropy approximation becomes
\begin{equation}
H \approx \log \frac{\sqrt{2 \pi} \sigma}{C(\vec{\varepsilon})}+C(\vec{\varepsilon}) \left(\frac{1}{2}+\varepsilon_{p}\mu_{(p)}^{G}-\frac{\varepsilon_{p}}{2 \sigma^{2}} \mu_{(p+2)}\right)
\end{equation}
Substituting (32) for $C(\vec{\varepsilon})$, this gives the result in (21), which is shown in Theorem 3's proof to be equivalent to (12).
\end{proof}

This theorem shows that the non-Gaussian effects of asymmetry in the probability density are at least a 2nd order effect in the entropy. Hence to study its effects on the entropy we must proceed to higher order in our approximations.

\section{Second-Order Approximations}
In this section we study the effects of (30)'s asymmetry on its entropy. As shown in the previous section, this requires going beyond a first order approximation. Here we proceed to second order.

\subsection{Odd $q$th and Even $p$th-Moment Perturbed Gaussian continued...}

\begin{theorem}
A 2nd order approximation of (30)'s moments and normalization function using (31) gives the following 2nd order result for (30)'s entropy.
\begin{eqnarray}
H_{q,p} & \approx & \log\frac{\sqrt{2\pi}\sigma}{C_{2}(\vec{\varepsilon})}+C_{2}(\vec{\varepsilon})\left(\frac{1}{2}+\varepsilon_{p}\sigma^{p}\left((p-1)!!-\frac{(p+1)!!}{2}\right)\right)\nonumber \\
 &  & +\varepsilon_{q}^{2}\sigma^{2q}\left(\frac{(2q+1)!!}{2}-(2q-1)!!\right)+\varepsilon_{p}^{2}\sigma^{2p}\left(\frac{(2p+1)!!}{2}-(2p-1)!!\right)\label{eq:37}
\end{eqnarray}
where 
\begin{equation}
C_{2}(\vec{\varepsilon})=\left(1-(p-1)!!\varepsilon_{p}\sigma^{p}+\frac{\varepsilon_{q}^{2}}{2}(2q-1)!!\sigma^{2q}+\frac{\varepsilon_{p}^{2}}{2}(2p-1)!!\sigma^{2p} \right)^{-1}.
\end{equation}
This approximation has the following upper bound on its error.
\begin{equation}
\left|\Delta H_{q,p}(2) \right| \leq \log\left(1+\gamma_{q,p}(0,2)\right)+\gamma_{q,p}(2,2)+\gamma_{q,p}(q,2)+ \gamma_{q,p}(p,2),
\end{equation}
where
\begin{equation}
\gamma_{q,p}(0,2)=\beta_{q,p}\left(\frac{(3p-1)!!}{6} \varepsilon_{p}^{3}\sigma^{3p}+\frac{(2q+p-1)!!}{2}\varepsilon_{q}^{2} \varepsilon_{p} \sigma^{2q+p}\right),
\end{equation}

\begin{equation}
\gamma_{q,p}(2,2)=\beta_{q,p}\left(\frac{(3p+1)!!}{12}\varepsilon_{p}^{3}\sigma^{3p}+\frac{(2q+2p+1)!!}{4} \varepsilon_{q}^{2}\varepsilon_{p}\sigma^{2q+p} \right),
\end{equation}

\begin{equation}
\gamma_{q,p}(q,2)=\beta_{q,p}\left(\frac{(4q-1)!!}{6}\varepsilon_{q}^{4} \sigma^{4q}+\frac{(2q+2p-1)!!}{2}\varepsilon_{q}^{2}\varepsilon_{p}^{2} \sigma^{2q+2p} \right),
\end{equation}

\begin{equation}
\gamma_{q,p}(p,2)=\beta_{q,p}\left(\frac{(4p-1)!!}{6}\varepsilon_{p}^{4}\sigma^{4p}+\frac{(2q+2p-1)!!}{2} \varepsilon_{q}^{2} \varepsilon_{p}^{2} \sigma^{2q+2p}\right),
\end{equation}

and

\begin{equation}
\beta_{q,p}=\left(\frac{1}{C_{2}(\vec{\varepsilon})}-\frac{(3p-1)!!}{6} \varepsilon_{p}^{3}\sigma^{3p}-\frac{(2q+p-1)!!}{2}\varepsilon_{q}^{2} \varepsilon_{p} \sigma^{2q+p}\right)^{-1}.
\end{equation}

\end{theorem}
\begin{proof}

We expand the series in (31) to second order. This gives 
$$
\mu_{(k)} \approx C(\vec{\varepsilon})\int_{-\infty}^{\infty}\frac{e^{-\frac{x^{2}}{2 \sigma^{2}}}}{\sqrt{2 \pi} \sigma}x^{k}\left[1+\varepsilon_{q}x^{q}-\varepsilon_{p}x^{p}+\frac{1}{2}\left(\varepsilon_{q}x^{q}-\varepsilon_{p}x^{p} \right)^{2} \right], $$
which can be simplified to 
\begin{equation}
\mu_{(k)} \approx C(\vec{\varepsilon})\int_{-\infty}^{\infty}\frac{e^{-\frac{x^{2}}{2 \sigma^{2}}}}{{\sqrt{2 \pi} \sigma}}x^{k}\left(1+\varepsilon_{q}x^{q}-\varepsilon_{p}x^{p}+\frac{1}{2} \varepsilon_{q}^{2} x^{2q}-\varepsilon_{q} \varepsilon_{p} x^{p+q}+\frac{1}{2} \varepsilon_{p}^{2} x^{2p} \right). \end{equation}
(45) with $k=0$ gives the equation for the second order approximate normalization function,
\begin{equation}
1 \approx C(\vec{\varepsilon})\left(1-\varepsilon_{p}\mu_{(p)}^{G}+\frac{1}{2} \varepsilon_{q}^{2} \mu_{(2q)}^{G}+\frac{1}{2} \varepsilon_{p}^{2} \mu_{(2p)}^{G} \right). \end{equation}
From (45), the second order approximation for the non-Gaussian moments is  
\begin{equation}
\mu_{(k)} \approx C(\vec{\varepsilon})\left(\mu_{(k)}^{G}-\varepsilon_{p}\mu_{(p+k)}^{G}+\frac{1}{2} \varepsilon_{q}^{2} \mu_{(2q+k)}^{G}+\frac{1}{2} \varepsilon_{p}^{2} \mu_{(2p+k)}^{G} \right) \end{equation}
for $k$ even, and it is 
\begin{equation}
\mu_{(k)} \approx C(\vec{\varepsilon})\left(\varepsilon_{q}\mu_{(q+k)}^{G}-\varepsilon_{q}\varepsilon_{p} \mu_{(p+q+k)}^{G} \right) \end{equation}
for $k$ odd.

Substituting (47) and (48) for the moments into (34), grouping terms, and neglecting all terms with factors of the form $\varepsilon_{j} \varepsilon_{k} \varepsilon_{l}$, gives the following result for the second order entropy approximation.
\begin{align}
H &\approx \log \frac{\sqrt{2 \pi}\sigma}{C(\vec{\varepsilon})}+C(\vec{\varepsilon})\times\nonumber\\
 &\left( \frac{1}{2}+\varepsilon_{p}\mu_{(p)}^{G} +\frac{1}{2\sigma^{2}}\left(\frac{\varepsilon_{q}^{2}}{2}\mu_{(2q+2)}^{G}+\frac{\varepsilon_{p}^{2}}{2}\mu_{(2p+2)}^{G}-\varepsilon_{p}\mu_{(p+2)}^{G} \right)-\varepsilon_{q}^{2}\mu_{(2q)}^{G}-\varepsilon_{p}^{2}\mu_{(2p)}^{G}\right)\label{eq:49}
\end{align}
Applying (46) to substitute for $C(\vec{\varepsilon})$ in (49), using (47) and (48) to substitute for the moments in (49), and grouping terms gives the results in (37)-(38).

We subtract (34) with the true moments from the same entropy expression with second order approximate moments. This gives after simplification
\begin{equation}
\Delta H_{q,p}(2)=\log \left(1 \pm C(\vec{\varepsilon})\left|\Delta_{0}(2)\right|\right) \pm \frac{1}{2 \sigma^{2}} \left|\Delta \mu_{(2)}(2) \right| \pm \varepsilon_{q} \left|\Delta \mu_{(q)}(2) \right| \pm \varepsilon_{p} \left|\Delta \mu_{(p)}(2) \right|. \end{equation}
The maximum value of (50) is given by
\begin{equation}
\left|\Delta H_{q,p}(2) \right| \leq \log \left(1 + C(\vec{\varepsilon})\left|\Delta_{0}(2)\right|\right) + \frac{1}{2 \sigma^{2}} \left|\Delta \mu_{(2)}(2) \right| + \varepsilon_{q} \left|\Delta \mu_{(q)}(2) \right| + \varepsilon_{p} \left|\Delta \mu_{(p)}(2) \right|
\end{equation}

We then note that given (31), we have the following upper bounds on the second order errors.
\begin{equation}
\left|\Delta_{0}(2) \right| \leq \int_{-\infty}^{\infty} \frac{e^{-\frac{x^{2}}{2\sigma^{2}}}}{\sqrt{2 \pi}\sigma} \frac{1}{3!}\left(\varepsilon_{q}x^{q}-\varepsilon_{p}x^{p}\right)^{3},
\end{equation}
and
\begin{equation}
\left|\Delta \mu_{(k)}(2) \right| \leq C(\vec{\varepsilon}) \int_{-\infty}^{\infty} \frac{e^{-\frac{x^{2}}{2\sigma^{2}}}}{\sqrt{2 \pi}\sigma} \frac{1}{3!}x^{k}\left(\varepsilon_{q}x^{q}-\varepsilon_{p}x^{p}\right)^{3}.
\end{equation}

Next we note that we can upper bound $C(\vec{\varepsilon})$ by noting that given a second order approximation, the minimum value of $1/C(\vec{\varepsilon})$ is given by

\begin{equation}
\frac{1}{C(\vec{\varepsilon})} \geq \frac{1}{C_{2}(\vec{\varepsilon})} -\left|\Delta_{0}(2)\right|.
\end{equation}

Simplifying and integrating (52), applying (8) for the Gaussian moments, substituting into (54), and rearranging gives $\beta_{q,p}$ as an upper bound for $C(\vec{\varepsilon})$. Doing the same simplification, integration, and application of (8) for (53) with $k=2$, $k=q$, and $k=p$, substituting $\beta_{q,p}$ for $C(\vec{\varepsilon})$, and substituting these upper bounds on the error into (51) gives the results in (39)-(44).

\end{proof}

Consider (37)-(38) with $q=3$ and $p=4$. Hence we are specifying non-Gaussian 3rd and 4th moments in (30). Note for this case from (47)-(48) that the \textit{skewness}, defined as $\mu_{(3)}/\mu_{(2)}^{3/2}$, is directly proportional to $\varepsilon_{3}$ at first order. Hence we refer to $\varepsilon_{3}$ as the skewness parameter. 

Now first let us set the skewness parameter $\varepsilon_{3}=0$ to see the second order effects that arise from the specification of a $4$th moment. This gives a second order approximation of the entropy of (1) with $p=4$. After simplification, (37) gives the result 
\begin{equation}
H \approx \log \sqrt{2 \pi} \sigma \left(1-3\varepsilon_{4} \sigma^{4}+52.5\varepsilon_{4}^{2} \sigma^{8}\right)+\frac{.5-4.5\varepsilon_{4}\sigma^{4}+367.5\varepsilon_{4}^{2}\sigma^{8}}{1-3\varepsilon_{4} \sigma^{4}+52.5\varepsilon_{4}^{2} \sigma^{8}}
\end{equation}
The first order entropy result in (26) and the second order entropy result in (55) are graphed above with their corresponding error bounds for $\sigma=1$ in Fig. 1. While there is a decrease in the entropy at first order, it is clear that there is a potentially large positive contribution to the entropy from the second order non-Gaussian effects. 

\begin{figure}
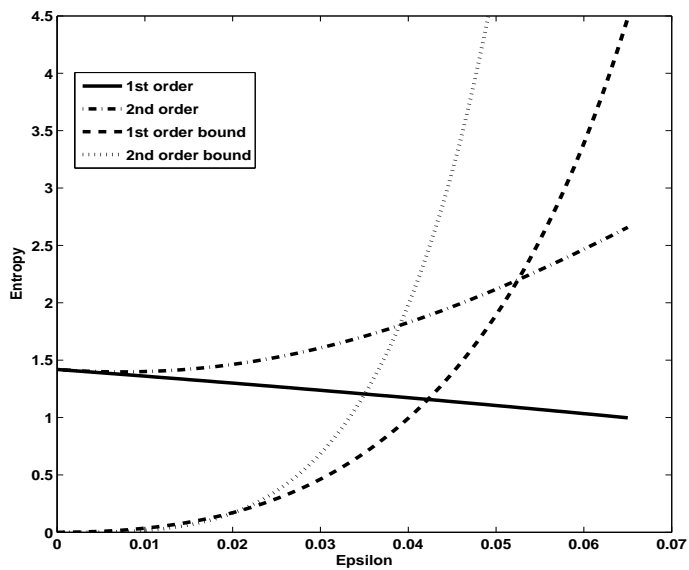


\figurebox{20pc}{25pc}{}[FigureCK.eps]
\caption{A graph of $H_{p}$ for $\varepsilon \in [0,.065]$, $p=4$, and $\sigma=1$. Solid is the 1st order approximation (26) with its corresponding error bound dashed. Dash-dotted is the 2nd order approximation (55) with corresponding error bound dotted.}
\label{fig:1}
\end{figure}

For nonzero skewness, $|\varepsilon_{3}| \geq 0$, and (37) with $q=3,p=4$ simplifies into 
\begin{equation}
H \approx \log \sqrt{2 \pi} \sigma \left(1-3\varepsilon_{4} \sigma^{4}+7.5\varepsilon_{3}^{2}\sigma^{6}+52.5\varepsilon_{4}^{2} \sigma^{8}\right)+\frac{.5-4.5\varepsilon_{4}\sigma^{4}+37.5\varepsilon_{3}^{2}\sigma^{6}+367.5\varepsilon_{4}^{2}\sigma^{8}}{1-3\varepsilon_{4} \sigma^{4}+7.5\varepsilon_{3}^{2}\sigma^{6}+52.5\varepsilon_{4}^{2} \sigma^{8}}.
\end{equation}
It is clear from (56) that this second order entropy approximation increases with the skewness parameter. Hence at second order, the skewness is increasing the entropy. Examination of the result in (37) shows that this is the case for general odd $q$. Hence in general, the lowest order effect of asymmetry on the entropy is to increase it. This could be foreseen by noting from (46)-(47) that the lowest order effect of $\varepsilon_{q}$ is to increase the variance and all even moments.

\section{Density Estimation}

Consider a univariate time-series $x_{t}$, $t \in \mathbb{Z}^{+}$, which is suspected to be near-Gaussian, and $M$ moment estimates of the form
\begin{equation}
\widehat{\mu}_{(m)}=\frac{1}{T}\sum_{t=1}^{T}\left(x_{t}\right)^{m}. \end{equation}
We suppose that our sample size $T$ is sufficiently large for the first $M$ moment estimates to be reasonably accurate. We use the empirical moments to estimate a corresponding, near-Gaussian, probability density function of the form
\begin{equation}
p(x)=\frac{C(\vec{\varepsilon})}{\sqrt{2 \pi} \sigma}e^{-\frac{x^{2}}{2 \sigma^{2}}}e^{\sum_{k \neq 2}^{M} \varepsilon_{k}x^{k}}, \end{equation}
where the $\varepsilon_{k}$ for even $k$ are negative and those for odd $k$ may be positive or negative. To estimate (58), we require estimates of the $M$ parameters $\sigma, \varepsilon_{1},\varepsilon_{3},...,\varepsilon_{M}$.  

\subsection{Procedure}
For the 1st step in our procedure, we choose a perturbative expansion of order $1$ with which to approximate (58), as 
\begin{equation}
p(x) \approx \frac{C(\vec{\varepsilon})}{\sqrt{2 \pi} \sigma}e^{-\frac{x^{2}}{2 \sigma^{2}}}\left(1+\sum_{k\neq 2}^{M} \varepsilon_{k}x^{k} \right). \end{equation}
(59) then gives an approximation of the normalization function $C(\vec{\varepsilon})$,
\begin{equation}
C_{(1)}(\vec{\varepsilon}) \approx \left(1+\sum_{k \neq 2}^{M} \varepsilon_{k} \mu_{(k)}^{G} \right)^{-1}.
\end{equation}
We then use (60) and set the empirically estimated moments in (57) equal to the corresponding perturbative approximations, as 
\begin{equation}
\widehat{\mu}_{(m)}=\frac{\mu_{(m)}^{G}+\sum_{k \neq 2}^{M} \varepsilon_{k} \mu_{(k+m)}^{G}}{1+\sum_{k \neq 2}^{M} \varepsilon_{k} \mu_{(k)}^{G}}=\frac{(m-1)!!\sigma^{m}+\sum_{k \neq 2}^{M} \varepsilon_{k} (k+m-1)!! \sigma^{k+m}}{1+\sum_{k \neq 2}^{M} \varepsilon_{k} (k-1)!! \sigma^{k}}. \end{equation}
Rearranging (61) gives a system of $M$ algebraic equations and $M$ unknowns. This system is then algebraically solved to give the resulting 1st order parameter estimates $\widehat{\sigma}^{(1)}, \widehat{\varepsilon}_{1}^{(1)},\widehat{\varepsilon}_{3}^{(1)},...,\widehat{\varepsilon}_{M}^{(1)}=\left[\widehat{\sigma}^{(1)},\vec{\varepsilon}^{\: (1)} \right]$. Note then that we must check that these estimates satisfy the perturbative hypothesis for all $k$,
\begin{equation}\widehat{\varepsilon}_{k} \widehat{\sigma}^{k} \ll1.\end{equation}
If (62) is not satisfied, then $x_{t}$ cannot be assumed to be near-Gaussian. Before proceeding, we note that based on the estimates $\left[\widehat{\sigma}^{(1)},\vec{\varepsilon}^{\: (1)} \right]$, there exists a set of integers $\left \{N^{(1)}(m) \right \}_{m=1}^{M}$ such that the upper bound on the error in the $m$th moment approximation, 
$$\left|\Delta \mu_{(m)}(N)\right| \leq C(\vec{\varepsilon})\int_{-\infty}^{\infty} \frac{e^{-\frac{x^{2}}{2 \widehat{\sigma}^{2}}}}{\sqrt{2 \pi} \widehat{\sigma}^{(1)}}\frac{x^{m}}{(N+1)!}\left(\sum_{k \neq 2}^{M}\widehat{\varepsilon}_{k}^{(1)}x^{k} \right)^{N+1}dx,$$
is minimized. These values of $\left \{N^{(1)}(m) \right \}_{m=1}^{M}$ can be found numerically or graphically, and for the $m$'s where $N^{(1)}(m) \geq 2$, we proceed to second order with our asymptotic series'.

The 2nd step in our procedure uses an expansion of the non-Gaussian component in (58) to order $2$, 
\begin{equation}
p(x) \approx \frac{C(\vec{\varepsilon})}{\sqrt{2 \pi} \sigma}\left(1+\sum_{k\neq 2}^{M} \varepsilon_{k}x^{k} + \frac{1}{2} \left(\sum_{k\neq 2}^{M} \varepsilon_{k}x^{k} \right)^{2} \right). \end{equation}
(63) gives a normalization function approximation $C_{(2)}(\vec{\varepsilon})$, and we use this to set our empirical moments in (57) equal to the 2nd order approximation,
\begin{equation}
\widehat{\mu}_{(m)}=C_{(2)}(\vec{\varepsilon}) \int_{-\infty}^{\infty} p_{G}(x) x^{m} \left(1+\sum_{k\neq 2}^{M} \varepsilon_{k}x^{k} + \frac{1}{2} \left(\sum_{k\neq 2}^{M} \varepsilon_{k}x^{k} \right)^{2} \right)dx, \end{equation}
where $p_{G}(x)=e^{-\frac{x^{2}}{2 \sigma^{2}}}/\sqrt{2 \pi}\sigma$. Integrating and rearranging (64) and using (8) gives another system of $M$ algebraic equations and $M$ unknowns, which is solved to get the second order parameter estimates $\widehat{\sigma}^{(2)}, \widehat{\varepsilon}_{1}^{(2)},\widehat{\varepsilon}_{3}^{(2)},...,\widehat{\varepsilon}_{M}^{(2)}=[\widehat{\sigma}^{(2)},\vec{\varepsilon}^{\: (2)}]$. We once again check the perturbative hypothesis in (62) with the second order estimates for all $k$. We then use the fact that the estimates $[\widehat{\sigma}^{(2)},\vec{\varepsilon}^{\: (2)}]$ imply a set of integers $\left \{N^{(2)}(m) \right \}_{m=1}^{M}$ such that the upper bounds on $\left|\Delta \mu_{(m)}(N) \right|$ are minimized. These values can be gotten numerically or graphically. For the $m$'s where $N^{(2)}(m) \geq 3$, we proceed to third order with our asymptotic expansions. 

These steps are repeated just as above until the step $n$ such that for all $m \in [1,M]$, $N^{(n)}(m) \leq n+1$. At this point, the order of maximum accuracy has been saturated for all the moment asymptotic expansions. The resulting estimates, denoted as $\widehat{\sigma}^{(*)}, \widehat{\varepsilon}_{1}^{(*)},\widehat{\varepsilon}_{3}^{(*)},...,\widehat{\varepsilon}_{M}^{(*)}=[\widehat{\sigma}^{(*)},\vec{\varepsilon}^{\: (*)}]$ are our final estimates for the background variance and non-Gaussianity parameters of the near-Gaussian density in (58).

\section{Discussion and Further Work}
We have outlined a perturbative technique for the calculation of entropic functionals for near-Gaussian densities. The method has been illustrated for several simple densities. These are of univariate maximum-entropy type in which certain moments of order greater than two are specified. The effect of these additional specified moments on the Shannon entropy are calculated at various orders in the developed perturbation expansion. In particular we show that to first order the entropy is decreased by the specification of a higher-order even moment. To second order, however, there is a possible increase. We also show that the effect of asymmetry (e.g. skewness) on the entropy is only apparent beyond first order, and at second order it increases the entropy.

In future work we plan a systematic analysis of general near-Gaussian maximum-entropy densities, both univariate and multivariate. We also develop general guidance for the optimal order of the perturbation expansion. 

As well as entropic functional calculation, we have also discussed a very inexpensive method for the calculation of the parameters of an approximating maximum entropy density given a set of presribed moments which are determined empirically. This was sketched in the previous section. In future work we intend to develop general expressions for the accuracy of the modelled moments at a given order of the perturbation expansion. We will also there apply the method to a practical application and compare our results to those obtained by a conventional convex optimization technique.

\bibliographystyle{biometrika}
\bibliography{refs,cv}

\end{document}